\documentclass{amsart}
\usepackage{graphicx}
\usepackage{times, diagrams, color, verbatim, bbm, mathrsfs, enumerate, array, longtable, amssymb, mathtools}
\pdfoutput=1
\title{A Universal Axiomatization of Metropolis-Rota Implication 
Algebras}
\author{Colin G. Bailey}
\address{School of Mathematics,  Statistics and Operations Research\\
Victoria University of Wellington\\
Wellington, New Zealand\\
}
\email{Colin.Bailey@vuw.ac.nz}
\author{Joseph S. Oliveira}
\address{
Pacific Northwest National Laboratories\\
Richland\\
U.S.A.}
\email{Joseph.Oliveira@pnl.gov}
\let\rsf\mathscr
\def\caret{\mathbin{\hat{\hphantom{m}}}}
\def\dblCaret{\mathbin{\hbox{$\hat{\hphantom{m}}$}\kern-10pt\hbox{$\hat{\hphantom{m}}$}}}
\def\eqcl[#1]{\pmb{[}#1\pmb{]}}
\def\one{\mathbf1}
\providecommand{\meet}{\mathbin{\wedge}}
\providecommand{\join}{\mathbin{\vee}}
\newcommand{\comp}[1]{\overline{#1}}
\ifx\upharpoonright\undefined
     \def\restrict{\hbox{\rm\kern0.166em\accent"12\kern-0.536em$\vert$\kern0.3em}}%
  \else
     \def\restrict{\upharpoonright}%
  \fi
\makeatletter
\def\twoSet#1#2{\left\{%
\vphantom{#2}#1\thinspace\right|\nolinebreak[3]\left.%
  #2%
  \vphantom{#1}%
  \right\}%
}
\def\oneSet#1{\left\lbrace#1\right\rbrace}

\newif\if@nstr
\def\setstrfalse{\let\if@nstr=\iffalse}
\def\setstrtrue{\let\if@nstr=\iftrue}
\def\@nstr #1#2{
\def\@@nstr ##1#1##2##3\@@nstr{\ifx
\@nstr ##2\setstrfalse \else \setstrtrue \fi }
\@@nstr #2#1\@nstr \@@nstr}
\def\@separate#1|#2@{\setFront{#1}\setBack{#2}}
\def\lb#1\rb{\@nstr|{#1} \if@nstr \@separate#1 @ \twoSet{\@setFront}{\@setBack}%
\else \@separate |{#1 }@ \oneSet{\@setBack}\fi%
}
\def\setFront#1{\def\@setFront{#1}}
\def\setBack#1{\def\@setBack{#1}}
\def\Set#1{\lb{#1}\rb}

\def\oneBrk#1{\left\langle#1\right\rangle}
\def\twoBrk#1#2{\left\langle%
\vphantom{#2}#1\thinspace\right|\nolinebreak[3]\left.%
  #2%
  \vphantom{#1}%
  \right\rangle%
}
\def\brk<#1>{\@nstr|{#1} \if@nstr \@separate#1 @ \twoBrk{\@setFront}{\@setBack}%
\else \@separate |{#1 }@ \oneBrk{\@setBack}\fi%
}

\def\thmref#1{\normalfont{theorem}~\ref{#1}}
\def\lemref#1{\normalfont{lemma}~\ref{#1}}

\theoremstyle{plain}
\newtheorem{thm}{Theorem}[section]
\newtheorem{lem}[thm]{Lemma}
\newtheorem{cor}[thm]{Corollary}
\newtheorem{prop}[thm]{Proposition}
\newtheorem{defn}[thm]{Definition}

\theoremstyle{remark}
{}
{\newtheorem{example}{Example}[section]}
{}
{}

\@ifpackageloaded{bbm}{

}{

}
\makeatother

\begin{document}

\begin{abstract}
	We show that the class of Metropolis-Rota implication 
	algebras can be given a universal axiomatization using an 
	operation closely related to composition in oriented matroids. 
	Lastly we describe the role of our new operation in the 
	collapse of an MR-algebra. 
\end{abstract}
\maketitle

\section{Introduction}
Metropolis-Rota implication algebras (MR-algebras) were first seen in 
\cite{MR:cubes} as an algebraic representation of the poset of faces of 
an $n$-cube. Therein was introduced the partial reflection operator $\Delta$ 
that was used in conjunction with the lattice structure to 
characterize these face lattices -- in the finite case. They proved 
the following theorem.
\begin{thm}\label{thm:MR}
Let $\mathcal L$ be a finite lattice with minimum $0$ and maximum $1$. For 
every $x\not=0$, let $\Delta_x$ be a function defined on the segment 
$[0, x]$ and taking values in $[0, x]$. Assume 
\begin{enumerate}[(i)]
	\item  If $a\le b\le x$ then $\Delta_x(a)\le \Delta_x(b)$;
	
	\item  $\Delta_x^2=id$ (the identity map);
	
	\item  Let $a<x$ and $b<x$. Then the following two conditions are 
	equivalent:
	\begin{gather*}
		\Delta_x(a)\join b<x\text{ and }a\meet b=0. 			
	\end{gather*}
\end{enumerate}
Then $\mathcal L$ is isomorphic to the lattice of faces of an $n$-cube, for 
some $n$. 

Conversely, if $\mathcal L$ is the lattice of faces of an $n$-cube, and 
$\Delta_x(y) $ is the antipodal face of $y$ within the face $x$, then 
$\mathcal L$ satisfies conditions (i) through (iii). 
\end{thm}

One can take an alternative approach to extending the Metropolis-Rota 
theorem by noting that the face lattice of the $n$-cube is exactly the 
lattice of closed intervals of the Boolean algebra $\bold 2^n$, and the 
$\Delta$ operator is induced by local complementation. One can then ask 
for an algebraic characterization of such lattices, which is provided 
by \cite{BO:eq},  wherein it was  shown that 
such lattices are essentially characterized by the Metropolis-Rota axioms plus 
atomicity and the fact that every interval $[x, y]$,  where $x>0$,  is a Boolean algebra. 

In studying such lattices one quickly discovers that $0$ plays a rather 
ambiguous role. It makes the structure into a lattice, but it prevents 
the class of such structures from being closed under cartesian products. 
It also makes the natural homomorphisms trivial (they are forced to be 
embeddings). 

The next obvious step then is to eliminate zero. The resulting 
structures are upper semi-lattices, and (as is clear from the 
Bailey-Oliveira theorem) they have a natural  implication structure. 
(This was first pointed out and studied  by Oliveira -- see 
\cite{Joe:thesis}.) Furthermore the class of such structures is now closed under 
cartesian products and  homomorphic images. 
We call the elements of this class {\em MR-algebras}. 

There are many ways to make this class into a variety--close under 
subobjects; add a new constant naming an atom; or add new operations. 

The second option is 
unsatisfactory, as the resulting category is naturally isomorphic to the 
category of Boolean algebras. The first 
option gives rise to \emph{cubic implication algebras} which are studied in 
detail in \cite{BO:eq}. A perhaps unfortunate side effect of the 
generality of cubic algebras is that there are many finite cubic 
algebras that are not face lattices of cubes,  but only embed into 
such lattices. In this paper we look at the last option,  
providing a new operation,  $\caret$, that simultaneously generalizes $\Delta$ 
and provides certain important lower bounds -- and so partly encodes 
the MR-axiom. From this we get a univeral axiomatization.
That this encoding is `successful' is illustrated by 
two things 
\begin{enumerate}[-]
	\item  all algebras with this new operation satisfying the axioms 
	below are exactly the cubic algebras satisfying the MR axiom;

	\item  and hence the finite ones are exactly the face lattices of 
	$n$-cubes.
\end{enumerate}

This new operation can be considered a variation on the 
well-known operation of composition in oriented matroids.

The paper begins with background information on cubic and MR-algebras 
as developed in \cite{BO:eq},  defines $\caret$ and shows that it 
captures the MR-axiom. We then give purely universal axioms for 
MR-algebras and show that all such algebras are cubic algebras with 
caret,  and hence are MR-algebras. 

The latter part of the paper develops the notion of \emph{collapse} of an 
MR-algebra and shows that caret is closely related to meet in an 
associated implication lattice. 

\section{Cubic \& MR algebras and the new operation}
First we  recall some definitions. 
\begin{defn}
	A \emph{cubic algebra} is a join semi-lattice with one and a binary 
	operation $\Delta$ satisfying the following axioms:
	\begin{enumerate}[a.]
		\item  $x\le y$ implies $\Delta(y, x)\join x = y$;
		
		\item  $x\le y\le z$ implies $\Delta(z, \Delta(y, x))=\Delta(\Delta(z, 
		y), \Delta(z, x))$;
		
		\item  $x\le y$ implies $\Delta(y, \Delta(y, x))=x$;
		
		\item  $x\le y\le z$ implies $\Delta(z, x)\le \Delta(z, y)$;
		
		\item[] Let $xy=\Delta(1, \Delta(x\join y, y))\join y$ for any $x$, $y$ 
		in $\mathcal L$. Then:
		
		\item  $(xy)y=x\join y$;
		
		\item  $x(yz)=y(xz)$;
	\end{enumerate}
\end{defn} 

\begin{defn}
	An \emph{MR-algebra} is a cubic algebra satisfying the MR-axiom:\\
	if $a, b<x$ then 
	\begin{gather*}
		\Delta(x, a)\join b<x\text{ iff }a\meet b\text{ does not exist.}
	\end{gather*}
\end{defn}

\begin{example}
    Let $X$ be any set,  and 
    $$
    \rsf S(X)=\Set{\brk<A, B> | A, B\subseteq X\text{ and }A\cap 
    B=\emptyset}.
    $$
    Elements of $\rsf S(X)$ are called \emph{signed subsets} of $X$.
    The operations are defined by 
    \begin{align*}
	1&=\brk<\emptyset,  \emptyset>\\
	\brk<A, B>\join\brk<C, D>&=\brk<A\cap C,  B\cap D>\\
	\Delta(\brk<A, B>, \brk<C, D>)&=\brk<A\cup D\setminus B, 
	B\cup C\setminus A>.
\end{align*}
These are all atomic MR-algebras.
\end{example}

\begin{example}
    Let $B$ be a Boolean algebra, then the \emph{interval algebra} of $B$ is
$$
\rsf I(B)=\Set{[a, b] | a\le b \text{ in }B}
$$
ordered by inclusion. The operations are defined by
\begin{align*}
	1&=[0, 1]\\
	[a, b]\join[c, d]&=[a\meet c, b\join d]\\
	\Delta([a, b], [c, d])&=[a\join(b\meet\comp d), b\meet(a\join\comp c)].
\end{align*}
These are all atomic MR-algebras. For further details see \cite{BO:eq}.

We note that $\rsf S(X)$ is isomorphic to $\rsf I(\wp(X))$. 
\end{example}

\begin{example}
    Let $B$ be a Boolean algebra and $\rsf F$ be a filter in $B$. Then 
the \emph{filter algebra} is the set
$$
\rsf I(\rsf F)=\Set{[a, b] | a\le b\text{ and }\comp a, b\in\rsf F}
$$
with the same order and operations as for interval algebras.

It is easy to show that if $\rsf F$ is a principal filter then 
$\rsf I(\rsf F)$ is isomorphic to an interval 
algebra. 
Hence every finite filter gives only interval algebras
and by the Metropolis-Rota theorem every finite MR-algebra is an 
interval algebra. 
\end{example}

\begin{defn}\label{def:caret}
	Let $\mathcal L$ be a cubic algebra. Then for any $x, y\in\mathcal L$ 
	we define the (partial) operation $\caret$ (caret) by:
	$$
		x\caret y=x\meet\Delta(x\join y, y)
	$$
	whenever this meet exists. 
\end{defn}

\begin{lem}
	In an MR-algebra the caret operation is total.
\end{lem}
\begin{proof}
	The meet will not exist iff $\Delta(x\join y, \Delta(x\join y, 
	y))\join x<x\join y$. 
	But the left-hand-side is exactly $x\join y$.
\end{proof}

To get the converse we need the following result from \cite{BO:eq} 
theorem 4.3. 
\begin{thm}\label{thm:cons}
	Let $\mathcal L$ be a cubic algebra.
	Let $a, b$ be in ${\mathcal L}$. There is no pair of elements 
	$x_1, x_2$ such that $a, b<x_1, x_2$ and 
	$$
	\Delta(x_1, a)\join b=x_1, \qquad\Delta(x_2, a)\join b<x_2. 
	$$
\end{thm}
Using this we can get the desired connection between caret and the MR 
axiom.
\begin{thm}
	Let $\mathcal L$ be a cubic algebra on which caret is total. Then 
	$\mathcal L$ satisfies the MR-axiom.
\end{thm}
\begin{proof}
	Let $a, b<x$.
	
	First,  if $a\meet b$ exists then we have
	$x\geq a\join\Delta(x, b)\geq(a\meet b)\join\Delta(x, a\meet b)=x$.
	
	Conversely suppose that $a\join\Delta(x, b)=x$. There are two cases 
	\begin{enumerate}[-]
			\item  if $a\join b$ is one of $a$ or $b$,  then $a$ and $b$ are 
			comparable and the meet clearly exists.
		
			\item  Otherwise $a, b<a\join b$. By \thmref{thm:cons} we must have 
			\begin{equation}\label{eq:one}
				a\join\Delta(a\join b, b)=a\join b.
			\end{equation}
			Then we have
			\begin{align*}			
				a\caret\Delta(a\join b, b)&=a\meet\Delta(a\join\Delta(a\join b, b), 
				\Delta(a\join b, b))&&\text{ by definition}\\
				&=a\meet\Delta(a\join b, \Delta(a\join b, b))&&\text{ by 
				\eqref{eq:one}}\\
				&=a\meet b.
			\end{align*}
		\end{enumerate}
\end{proof}

\section{Axiomatics}
We turn now to providing an axiomatic description of cubic 
algebras with caret. The first version gives a caret-like operation 
that is enough to give all MR-algebras,  but not strong enough to 
prove that caret satisfies definition \ref{def:caret}. Following this 
result we consider how to improve the fit. 

The fact that the axioms provided are universal shows that the class 
of MR-algebras forms a variety.

\begin{thm}
	Let $\mathcal L$ be a join-semilatice with $1$ and a binary 
	operation  $\caret $ (caret) satisfying the following axioms:
	\begin{enumerate}[(a)]
			\item  $\forall x, y$\qquad $x\join(y\caret x)=x\join y$;
		
			\item  $\forall x, y$\qquad $(1\caret x)\caret(1\caret 
			y)=1\caret(x\caret y)$;
		
			\item  $\forall x$\qquad $1\caret(1\caret x)=x$;
		
			\item  $\forall x, y$\qquad if $x\le y$ then $1\caret x\le 1\caret y$;
		
			\item  If we define $x\rightarrow y$ as equal to 
			$y\join(1\caret(x\caret y))$ then
			\begin{enumerate}[i)]
							\item  $(x\rightarrow y)\rightarrow y=x\join y$;
						
							\item  $x\rightarrow(y\rightarrow z)=y\rightarrow(x\rightarrow z)$;
						\end{enumerate}
		
			\item  $\forall x, y$\qquad $(x\join y)\rightarrow((x\join y)\caret 
			y)=1\caret(x\rightarrow y)$;
		
			\item  $\forall x, y$\qquad $x\caret y\le (x\join y)\caret y$;
		
			\item  $\forall x, y$\qquad $x\caret y\le x$.
		\end{enumerate}
		Then $\mathcal L$ is an MR-algebra.
\end{thm}
\begin{proof}
	We proceed via a series of lemmata. We aim to show that with these 
	axioms we can define a $\Delta$-operator that makes $\mathcal 
	L$ into a cubic algebra on which the operation 
	$\brk<a, b>\mapsto a\meet\Delta(a\join b, b)$ is total. From 
	the results above this implies that $\mathcal L$ is an 
	MR-algebra. However it is not enough to show that the caret 
	operations are the same,  for that we need one more axiom as 
	we see in \lemref{lem:extra}. This is given in 
	\thmref{thm:extra}. 
	
	The next four lemmas establish that $\brk<\mathcal L, \to>$ 
	is an implication algebra,  by checking each of the axioms in 
	turn.
	
\begin{lem}
	$(x\rightarrow y)\join x=1$.
\end{lem}
\begin{proof}
	We recall that $(x\rightarrow y)\join x=(x\join y)\join 
	1\caret(x\caret y)$. Also we have 
	\begin{align*}
		x\caret y\le(x\join y)\caret y&\le x\join y&&\text{ by (g, h)}\\
		\text{Therefore }(x\rightarrow y)\join x&\geq (x\caret y)\join 
		1\caret(x\caret y) \\
		&=1\join (x\caret y)&&\text{ by (a)}\\
		&=1.
		\end{align*}
\end{proof}

\begin{lem}
	$(x\rightarrow y)\caret x\le 1\caret x$.
\end{lem}
\begin{proof}
	By axioms (c) and (d) the lemma is true iff $1\caret((x\rightarrow y)\caret x)\le x$.
	\begin{align*}
			1\caret((x\rightarrow y)\caret x)&=(1\caret(x\rightarrow 
			y))\caret(1\caret x)\\
			&\le \bigl((1\caret(x\rightarrow y))\join(1\caret 
			x)\bigr)\caret(1\caret x)&&\text{ by (g)}\\
			(1\caret(x\rightarrow y))\join(1\caret x)&=1\caret((x\rightarrow 
			y)\join x)\\
			&=1\caret 1=1.\\
			\intertext{Hence }
			1\caret((x\rightarrow y)\caret x)&\le 1\caret(1\caret x)=x.
		\end{align*}
\end{proof}

\begin{lem}
	$(x\rightarrow y)\rightarrow x=x$.
\end{lem}
\begin{proof}
	By definition of $\rightarrow$,  the left-hand-side is greater than 
	$x$.
	
	By the definition of $\rightarrow$ and the proof of the last lemma 
	we have 
	$(x\rightarrow y)\rightarrow x=x\join1\caret((x\rightarrow y)\caret 
	x)\le x\join x=x$.
\end{proof}

\begin{prop}
	$\brk<\mathcal L, \rightarrow>$ is an implication algebra.
\end{prop}
\begin{proof}
	This clear,  as the last lemma and axiom (e) give the axioms for 
	implication algebras.
\end{proof}

Now we turn to the verification that we have a cubic algebra by first 
defining $\Delta$ and then checking each of the remaining axioms. 

\begin{defn}
	Let $a\geq b$. Then $\Delta(a, b)=a\caret b$.
\end{defn}

\begin{lem}
	Let $a\geq b$. Then $\Delta(a, \Delta(a, b))=b$.
\end{lem}
\begin{proof}
	By (h) we know that $\Delta(a, b)\le a$. By (f) we have $\Delta(1, 
	a\rightarrow b)=1\caret(a\rightarrow b)=a\rightarrow(a\caret 
	b)=a\rightarrow\Delta(a, b)$. As $a\rightarrow\Delta(a, b)\geq\Delta(a, 
	b)$ and $\rightarrow$ is an implication operation we have
	$a\meet(a\rightarrow\Delta(a, b))=\Delta(a, b)=a\meet\Delta(1, 
	a\rightarrow b)$. Hence we have
	\begin{align*}
			\Delta(a, \Delta(a, b))&=a\meet\Delta(1, a\rightarrow\Delta(a, 
			b))\\
		&=a\meet\Delta(1, \Delta(1, a\rightarrow b))\\
		&=a\meet(a\rightarrow b)\\
		&=b.
	\end{align*}
\end{proof}

\begin{lem}
	Let $c\le b\le a$. Then $\Delta(a, c)\le\Delta(a, b)$.
\end{lem}
\begin{proof}
	Since $\rightarrow$ is an implication operation and $c\le b\le a$ we 
	have 
	$a\rightarrow c\le b\rightarrow c$ and $a\rightarrow 
	b=b\join(a\rightarrow c)$.Thus
	\begin{align*}
		\Delta(a, b)&=a\meet\Delta(1, a\rightarrow b)\\
		&=a\meet\Delta(1, b\join(a\rightarrow c))\\
		&=a\meet(\Delta(1, b)\join\Delta(1, a\rightarrow c))\\
		&\geq a\meet\Delta(1, a\rightarrow c)\\
		&=\Delta(a, c).
	\end{align*}
\end{proof}

\begin{lem}
	If $c\le b\le a$ then $\Delta(b, c)=b\meet\Delta(a, (b\rightarrow 
	c)\meet a)$.
\end{lem}
\begin{proof}
	First we note that 
		\begin{align*}
			a\rightarrow((b\rightarrow c)\meet a)&=((b\rightarrow c)\meet 
			a)\join(a\rightarrow c)&&\text{as }(b\rightarrow c)\meet a\geq c\\
			&=\bigl[(b\rightarrow c)\join(a\rightarrow 
			c)\bigr]\meet\bigl[a\join(a\rightarrow)\bigr]&&\text{in }[c, 1]\\
			&=(b\rightarrow c)\meet 1\\
			&=b\rightarrow c.\\
			\intertext{From this we can conclude that }
			\Delta(a, (b\rightarrow c)\meet a)&=a\meet\Delta(1, 
			a\rightarrow((b\rightarrow c)\meet a))\\
			&=a\meet\Delta(1, b\rightarrow c).\\
			\intertext{And so we get }
			\Delta(b, c)&=b\meet\Delta(1, b\rightarrow c)\\
			&=(b\meet a)\meet\Delta(1, b\rightarrow c)\\
			&=b\meet\Delta(a, (b\rightarrow c)\meet a).
		\end{align*}
\end{proof}

\begin{lem}
	If $c\le b\le a$ then $\Delta(a, \Delta(b, c))=\Delta(\Delta(a, b), 
	\Delta(a, c))$.
\end{lem}
\begin{proof}
	From the last lemma we have 
	\begin{align*}
			\Delta(a, \Delta(b, c))&=\Delta(a, b\meet\Delta(a, (b\rightarrow 
			c)\meet a))\\
			&=\Delta(a, b)\meet\Delta(a, \Delta(a, (b\rightarrow c)\meet a))\\
			&=\Delta(a, b)\meet (b\rightarrow c)\meet a.\\
			\intertext{The other side gives}
			\Delta(\Delta(a, b), \Delta(a, c))&=\Delta(a, b)\meet\Delta(a, 
			(\Delta(a, b)\rightarrow\Delta(a, c))\meet a).
		\end{align*}
		$(\Delta(a, b)\rightarrow\Delta(a, c))\meet a$ is the unique 
		complement of $\Delta(a, b)$ in $[\Delta(a, c), a]$ and $\Delta(a, -)$ 
		is an automorphism on $]\leftarrow, a]$ so its image under 
		$\Delta(a, -)$ is the unique complement of $\Delta(a, \Delta(a, b))=b$ 
		in $[c, a]$. But this is $(b\rightarrow c)\meet a$. Hence
		$(b\rightarrow c)\meet a = \Delta(a, 
			(\Delta(a, b)\rightarrow\Delta(a, c))\meet a)$.
\end{proof}

\begin{lem}
	If $b\le a$ then $a\rightarrow b=b\join\Delta(1, \Delta(a, b))$.
\end{lem}
\begin{proof}
	This is immediate from the definitions.
\end{proof}

\begin{cor}
	For any $a, b $ we have $a\rightarrow b=b\join\Delta(1, 
	\Delta(a\join b, b))$.
\end{cor}
\begin{proof}
	Since $a\rightarrow b = (a\join b)\rightarrow b$.
\end{proof}

It now follows from the above lemmas that the axioms of a cubic 
algebra are satisfied. To show that we have an MR-algebra it suffices 
to show that the `new' operation
\begin{equation}
	a\vartriangle b=a\meet\Delta(a\join b, b)
	\label{eq:caret}
\end{equation}
is total. 

In order to see that this is so,  it suffices to note that
$a\caret b\le a$ and $a\caret b\le\Delta(a\join b, b)$ and so
there is a lower bound to $a$ and $\Delta(a\join b, b)$. Since the 
algebra is an implication algebra we know that the meet always exists.
\end{proof}

Note that we have not proven that $\vartriangle$ and $\caret$ are 
the same operation. Indeed we cannot do so as the next lemma shows, 
although axioms (e-h) put strong constraints on the possibilities.

\begin{lem}\label{lem:extra}
	Let $\mathcal L$ be an MR-algebra, and for all
	$x, y$ let $p(x, y)$ be any element of $\mathcal L$ satisfying
	\begin{enumerate}[(a)]
			\item  if $x\geq y$ then $p(x, y)=\Delta(x, y)$;
			
			\item  $p(x, y)\le x$;
		
			\item  $p(x, y)\le \Delta(x\join y, y)$.
		\end{enumerate}
	Then $p$ is a caret operation on $\mathcal L$.
\end{lem}
\begin{proof}
	It suffices to show that $y\join\Delta(\one , \Delta(x\join y, y))= y\join\Delta(\one , p(x, y))$ for any 
	such $p$.
	Since every cubic algebra embeds into an interval algebra we may do 
	all of the necessary computations in an interval algebra.
	\begin{align*}
		y\join\Delta(\one , \Delta(x\join y, y))&=
		[y_{0}\meet \comp x_{0}, y_{1}\join\comp x_{1}]\\
		p(x, y)=[p_{0}, p_{1}]&\le x\meet\Delta(x\join y, y)\\
		&=[x_{0}\join(x_{1}\meet\comp y_{1}), x_{1}\meet(x_{0}\join\comp 
		y_{0})]\\
		\intertext{Therefore}
		x_{0}\join(x_{1}\meet\comp y_{1})\le p_{0}&\le p_{1}\le x_{1}\meet(x_{0}\join\comp 
		y_{0})&&(*)\\
		y\join\Delta(\one , p(x, y))&=[y_{0}\meet\comp p_{1}, y_{1}\join\comp 
		p_{0}]\\
		y_{0}\meet\comp p_{1}&\le y_{0}\meet\comp x_{0}\meet(\comp 
		x_{1}\join y_{1})&&\text{ by }(*)\\
		&=y_{0}\meet\comp x_{0}&&\text{ as }y_{0}\le y_{1}\\
		y_{0}\meet\comp p_{1}&\geq y_{0}\meet(\comp
				x_{1}\join(\comp x_{0}\meet	y_{0}))&&\text{ by }(*)\\
		&=(y_{0}\meet\comp x_{1})\join(y_{0}\meet\comp x_{0})\\
		&=y_{0}\meet\comp x_{0}&&\text{ as }x_{0}\le x_{1}.\\
		\intertext{Likewise}
		y_{1}\join \comp p_{0}&=y_{1}\join\comp x_{1}.
	\end{align*}
\end{proof}

This lemma shows that we need to add an additional axiom in order to 
ensure that caret is definable by \eqref{eq:caret}. 

\begin{thm}\label{thm:extra}
	Suppose that $\mathcal L$ is as in the above theorem, satisfying the 
	additional axiom
	\begin{enumerate}[i)]
		\item[\textup{(i)}] $(x\join y)\caret y\le x\rightarrow(x\caret y)$
	\end{enumerate}
	then $\mathcal L$ is an MR-algebra and $x\caret y=x\meet\Delta(x\join 
	y, y)$ for all $x, y$.
\end{thm}
\begin{proof}
	We only need to prove the last statement, and it is clear as
	$x\caret y\le x\meet ((x\join y)\caret y)= x\meet\Delta(x\join y, 
	y)\le x\meet (x\rightarrow(x\caret y))=x\caret y$.
\end{proof}

Our results show that the class of MR-algebras form a variety that is 
contained in the variety of cubic algebras. 

Any finite object in this variety satifies the hypotheses of 
\thmref{thm:MR} and is therefore isomorphic to the face lattice of 
an  $n$-cube.

\section{What is caret?}
The $\Delta$ operator on finite MR-algebras is very natural -- 
$\Delta(x, y)$ is the reflection of $y$ through the centre of the 
face $x$. But what of caret? 

\subsection{The signed set case}
Earlier we gave the example of the MR-algebra of signed sets. It is 
well-known that oriented matroids arise as subposets of $\rsf S(X)$, 
albeit with the reverse order,  but meet does not usually correspond 
to our join. However there is a close connection between composition 
and caret as we see in the next theorem. 
First we recall the definition of composition as used in  
$\rsf S(X)$. For notational convenience we write a signed set
$A$ as the pair $\brk<A^{+},  A^{-}>$. 

\begin{defn}\label{def:compositionOM}
	Let $A=\brk<A^{+}, A^{-}>$ and $B=\brk<B^{+}, B^{-}>$ be two signed
	subsets of $X$. Then the \emph{composition} of $A$ and $B$ is
	$$
		A\circ B=\brk<A^{+}\cup(B^{+}\setminus A^{-}),
A^{-}\cup(B^{-}\setminus A^{+})>.
	$$
\end{defn}

\begin{thm}
    Let $A$ and $B$ be two signed sets in $\rsf S(X)$. Then
    $$
    A\circ B= A\caret (\Delta B).
    $$
\end{thm}
\begin{proof}
    \begin{align*}
        A\caret B & = A\meet \Delta(A\join B,  B)  \\
        A\join B & =\brk<A^{+}\cap B^{+},  A^{-}\cap B^{-}>  \\
	\Delta(A\join B,  B) & = \brk<(A^{+}\cap B^{+})\cup 
	B^{-}\setminus(A^{-}\cap B^{-}),  (A^{-}\cap B^{-})\cup 
	B^{+}\setminus(A^{+}\cap B^{+})>  \\
         & = \brk<(A^{+}\cap B^{+})\cup 
	B^{-}\setminus A^{-},  (A^{-}\cap B^{-})\cup 
	B^{+}\setminus A^{+}>   \\
        A\caret B & = \brk<A^{+} \cup 
	B^{-}\setminus A^{-},  A^{-} \cup 
	B^{+}\setminus A^{+}>\\
	\intertext{Therefore }
	A\caret \Delta B &= \brk<A^{+} \cup 
	B^{+}\setminus A^{-},  A^{-} \cup 
	B^{-}\setminus A^{+}>\\
	&= A\circ B.
    \end{align*}
\end{proof}

In a later paper (\cite{BO:comp}) we study this connection between MR-algebras and 
oriented matroids in greater depth. 

\subsection{The general case}
This gives us some idea about caret. But more can be seen by 
considering the collapse of an MR-algebra. To get this collapse we 
define two relations on $\mathcal L$ that give us the collapsing 
relation. 

\begin{defn}
	Let $\mathcal L$ be a cubic algebra and $a, b\in\mathcal L$. Then
	\begin{align*}
		a\preceq b &\text{ iff }\Delta(a\join b, a)\le b\\
		a\simeq b &\text{ iff }\Delta(a\join b, a)=b.
	\end{align*}
\end{defn}

\begin{lem}
	Let $\mathcal L$, $a$, $b$ be as in the definition. Then
	$$
	a\preceq b\text{ iff }b=(b\join a)\meet(b\join\Delta(1, a)).
	$$
\end{lem}
\begin{proof}
	See \cite{BO:eq} lemmas 2.7 and 2.12.
\end{proof}

Also from \cite{BO:eq} (lemma 2.7c for transitivity) we know that $\simeq$ is an equivalence 
relation. In general it is not a congruence relation, but it does fit 
well with caret. In fact much more is true -- the structure $\mathcal 
L/\simeq$ is naturally an implication lattice. To show this we need 
to show that certain operations cohere with $\simeq$. 

However before doing so we show that this relation actually describes 
a natural property of intervals of Boolean algebras.
\begin{defn}
	Let $x=[x_{0}, x_{1}]$ be any interval in a Boolean algebra $B$. 
	Then the \emph{length} of $x$ is $x_{0}\join\comp x_{1}=\ell(x)$.
\end{defn}

\begin{lem}
	Let $b, c$ be intervals in a Boolean algebra $B$. Then 
	$$
	b\simeq c\iff\ell(b)=\ell(c).
	$$
\end{lem}
\begin{proof}
	From $b\simeq c$ we have 
	\begin{align*}
		\Delta(b\join c, c)&=[(b_{0}\meet c_{0})\join(\comp b_{1}\meet c_{1}), 
		(b_{1}\join c_{1})\meet(\comp b_{0}\join c_{0})]\\
		&=[c_{0}, c_{1}].\\
		\intertext{Hence}
		\ell(c)&= c_{0}\join\comp c_{1}\\
		&=\bigl((b_{0}\meet c_{0})\join(\comp b_{1}\meet 
		c_{1})\bigr)\join\comp{%
		\bigl((b_{1}\join c_{1})\meet(\comp b_{0}\join c_{0})\bigr)}\\
		&=(b_{0}\meet c_{0})\join(\comp b_{1}\meet 
		c_{1})\join(\comp b_{1}\meet\comp c_{1})\join(b_{0}\meet\comp 
		c_{0})\\
		&=b_{0}\join\comp b_{1}\\
		&=\ell(b).
	\end{align*}
	Comversely, if $\ell(b)=\ell(c)$ then we have
	\begin{align*}
		\Delta(b\join c, c)&=[(b_{0}\meet c_{0})\join(\comp b_{1}\meet c_{1}), 
		(b_{1}\join c_{1})\meet(\comp b_{0}\join c_{0})]\\
		\comp b_{1}\meet c_{1}&=c_{1}\meet(\comp b_{1}\join 
		b_{0})\meet(\comp b_{1}\join \comp b_{0})\\
		&=c_{1}\meet(\comp c_{1}\join 
		c_{0})\meet\comp b_{0}\\
		&=c_{0}\meet\comp b_{0}\\
		\intertext{and therefore }
		(b_{0}\meet c_{0})\join(\comp b_{1}\meet c_{1})&=
		(b_{0}\meet c_{0})\join(\comp b_{0}\meet c_{0})\\
		&=c_{0}.
	\end{align*}
	Likewise $(b_{1}\join c_{1})\meet(\comp b_{0}\join c_{0})=c_{1}$.
\end{proof}

First caret.

\begin{lem}\label{lem:caretComm}
	Let $\mathcal L$, $a$, $b$ be as in the definition. Then
	$$
	a\caret b\simeq b\caret a.
	$$
\end{lem}
\begin{proof}
	\begin{align*}
		\Delta(a\join b, a\caret b)&=\Delta(a\join b, a\meet\Delta(a\join b, 
		b))\\
		&=\Delta(a\join b, a)\meet\Delta(a\join b, \Delta(a\join b, 
		b))\\
		&=\Delta(a\join b, a)\meet b\\
		&=b\caret a.\\
		\intertext{It follows that $(a\caret b)\join (b\caret a)=a\join b$ and hence}
		\Delta((a\caret b)\join (b\caret a), a\caret b)&=\Delta(a\join b, 
		a\caret b)\\
		&=b\caret a.
	\end{align*}		
\end{proof}

\begin{lem}
	Let $\mathcal L$ be an MR algebra, and $a, \ b, \ c\in\mathcal L$ 
	with $b\simeq c$. Then $a\caret b\simeq a\caret c$.
\end{lem}
\begin{proof}
	Since we are working in an MR-algebra we may assume that $a, b, c$ 
	are all in an interval algebra,  so let
	$a=[a_{0}, a_{1}], b=[b_{0}, b_{1}], $ and $c=[c_{0}, c_{1}]$. Then
	\begin{align*}
		a\caret b&=[a_{0}\join(a_{1}\meet\comp b_{1}), 
		a_{0}\join(a_{1}\meet\comp b_{0})]=[s_{0}, s_{1}], \\
		a\caret c&=[a_{0}\join(a_{1}\meet\comp c_{1}), 
		a_{0}\join(a_{1}\meet\comp c_{0})]=[t_{0}, t_{1}]\\
		\text{ and }\qquad
		\Delta(b\join c, b)&=[(b_{0}\meet c_{0})\join(b_{1}\meet\comp c_{1}), 
		(b_{1}\join c_{1})\meet(b_{0}\join \comp c_{0})]\\
		&=[c_{0}, c_{1}]&&\text{ as }b\simeq c\\
		\Delta((a\caret b)\join(a\caret c), a\caret b)&=[(s_{0}\meet 
		t_{0})\join(s_{1}\meet\comp t_{1}), 
		(s_{1}\join t_{1})\meet(s_{0}\join \comp t_{0})]\\
		s_{0}\meet t_{0}&=(a_{0}\join(a_{1}\meet\comp 
		b_{1}))\meet(a_{0}\join(a_{1}\meet\comp c_{1}))\\
		&=a_{0}\join(a_{1}\meet\comp b_{1}\meet\comp c_{1})\\
		s_{1}\meet \comp t_{1}&=a_{1}\meet(a_{0}\join\comp b_{0})\meet\comp 
		a_{0}\meet(\comp a_{1}\join c_{0})\\
		&=a_{1}\meet\comp a_{0}\meet\comp b_{0}\meet c_{0}\\
		\text{Therefore }\qquad 
		(s_{0}\meet t_{0})\join(s_{1}\meet\comp t_{1})&=
		a_{0}\join(a_{1}\meet\comp b_{1}\meet\comp c_{1})\join(a_{1}\meet\comp 
		b_{0}\meet c_{0})\\
		&=a_{0}\join(a_{1}\meet[(\comp b_{0}\meet c_{0})\join(\comp 
		b_{1}\meet\comp c_{1})])\\
		&=a_{0}\join(a_{1}\meet\comp{[(b_{0}\join\comp  c_{0})\meet(
				b_{1}\join c_{1})]})\\
		&=a_{0}\join(a_{1}\meet\comp c_{1})\\
		\intertext{Likewise (essentially dually) we have}
		(s_{1}\join t_{1})\meet(s_{0}\join \comp 
		t_{0})&=a_{0}\join(a_{1}\meet\comp c_{0}).
	\end{align*}
\end{proof}

\begin{lem}
	Let $\mathcal L$ be an MR algebra, and $a, \ b, \ c, \ d\in\mathcal L$ 
	with $a\simeq d$ and $b\simeq c$. Then $a\caret b\simeq d\caret c$.
\end{lem}
\begin{proof}
	\begin{gather*}
		a\caret b\simeq a\caret c
		\simeq c\caret a
		\simeq c\caret d
		\simeq d\caret c.
	\end{gather*}
\end{proof}

A similar proof shows us that caret is associative mod $\simeq$. We 
leave this to the interested reader.
\begin{lem}
	Let $\mathcal L$ be an MR algebra, and $a, \ b, \ c\in\mathcal L$. 
	Then $a\caret (b\caret c)\simeq (a\caret b)\caret c$.
\end{lem}

Now we consider the second operation that will give rise to joins on 
$\mathcal L/\simeq$:
$$
a*b=a\join\Delta(a\join b, b).
$$
\begin{lem}
	$a*b\simeq b*a$.
\end{lem}
\begin{proof}
	We note that $\Delta(a\join b, a*b)=b*a$ and so we proceed as in 
	\lemref{lem:caretComm}.
\end{proof}

\begin{lem}
	Let $b\simeq c$. Then
	$a*b = a*c$ and $b*a = c*a$.
\end{lem}
\begin{proof}
	As $b\simeq c$ we have $\ell(b)=\ell(c)$.
	Therefore we have
	\begin{align*}
	a*b=a\join\Delta(a\join b, b)&=[a_{0}\meet(b_{0}\join\comp b_{1}), 
	a_{1}\join(b_{1}\meet\comp b_{0})]\\
	&=[a_{0}\meet\ell(b), a_{1}\join\comp{\ell(b)}]\\
	&=[a_{0}\meet\ell(c), a_{1}\join\comp{\ell(c)}]\\
	&=a\join\Delta(a\join c, c)=a*c.
	\end{align*}	
	The second result is left to the interested reader.
\end{proof}


\begin{thm}
	Let $a\simeq d$ and $b\simeq c$. Then
	$$
	a*b\simeq d*c.
	$$
\end{thm}
\begin{proof}
	Since $a*b=a*c\simeq c*a=c*d\simeq d*c$.
\end{proof}

As with caret we can also establish associativity,  but this time we 
have true associativity. 
\begin{thm}
	For all $a, b, c$
	$$
	a*(b*c) = (a*b)*c.
	$$
\end{thm}

Lastly we have to define an implication operation, and show that it 
satisfies the axioms for an implication operation.
For any $a, b$ we let
$$
a\Rightarrow b=\Delta(a\join b, a)\rightarrow b.
$$
A direct computation on intervals gives us
$$
a\Rightarrow b=[b_{0}\meet\comp{\ell(a)}, b_{1}\join\ell(a)].
$$

First this goes through $\simeq$:
\begin{lem}
	Let $b\simeq c$. Then $b\Rightarrow a=c\Rightarrow a$ and $a\Rightarrow b\simeq a\Rightarrow c$.
\end{lem}
\begin{proof}
	\begin{align*}
		b\Rightarrow a&=[a_{0}\meet\comp{\ell(b)}, a_{1}\join\ell(b)]\\
		&=[a_{0}\meet\comp{\ell(c)}, a_{1}\join\ell(c)]&&\text{ as }b\simeq 
		c\\
		&=c\Rightarrow a.
	\end{align*}

Let $\alpha=\ell(a)$. Then
	\begin{align*}
		a\Rightarrow b&=[b_{0}\meet\comp{\alpha}, b_{1}\join\alpha]\\
		a\Rightarrow c&=[c_{0}\meet\comp{\alpha}, c_{1}\join\alpha]\\
		(a\Rightarrow b)\join(a\Rightarrow c)&=
		[c_{0}\meet b_{0}\meet\comp{\alpha}, c_{1}\join 
		b_{1}\join\alpha]=[x_{0}, x_{1}]\\
		\Delta((a\Rightarrow b)\join(a\Rightarrow c), a\Rightarrow b)&=
		[x_{0}\join(x_{1}\meet\comp b_{1}\meet\comp\alpha), 
		x_{1}\meet(x_{0}\join\comp b_{0}\join\alpha)]\\
		x_{1}\meet\comp b_{1}\meet\comp\alpha&=(b_{1}\join 
		c_{1}\join\alpha)\meet\comp b_{1}\meet\comp\alpha\\
		&=c_{1}\meet\comp b_{1}\meet\comp\alpha\\
		\text{Therefore }x_{0}\join(x_{1}\meet\comp b_{1}\meet\comp\alpha)&=
		(b_{0}\meet c_{0}\meet\comp\alpha)\join(c_{1}\meet\comp 
		b_{1}\meet\comp\alpha)\\
		&=\comp\alpha\meet((b_{0}\meet c_{0})\join(c_{1}\meet\comp b_{1}))\\
		&=\comp\alpha\meet c_{0}&&\text{ as }b\simeq c.
	\end{align*}
	Likewise we have $x_{1}\meet(x_{0}\join\comp 
	b_{0}\join\alpha)=c_{1}\join\alpha$.
\end{proof}

Now we need to check that the axioms for implication work:
\begin{thm}
	Let $a, b, c$ be arbitrary. Then
	\begin{enumerate}[(a)]
			\item  $(a\Rightarrow b)\Rightarrow a=a$;
		
			\item  $(a\Rightarrow b)\Rightarrow b = b*a$;
		
			\item  $a\Rightarrow (b\Rightarrow c)=b\Rightarrow (a\Rightarrow c)$.
		\end{enumerate}
\end{thm}
\begin{proof}
	Let $\alpha=\ell(a)$ and $\beta=\ell(b)$.
	\begin{enumerate}[(a)]
			\item   
				\begin{align*}
					(a\Rightarrow b)\Rightarrow a&=[b_{0}\meet\comp{\alpha}, 
					b_{1}\join\alpha]\Rightarrow a\\
					\ell(a\Rightarrow b)&=
					(b_{0}\meet\comp{\alpha})\join\comp{b_{1}\join\alpha}\\
					&=(b_{0}\meet\comp\alpha)\join(\comp b_{1}\meet\comp\alpha)\\
					&=\ell(b)\meet\comp\alpha.\\
					\text{Thus }
					[b_{0}\meet\comp{\alpha}, 
					b_{1}\join\alpha]\Rightarrow a&=
					[a_{0}\meet\comp{\beta\meet\comp\alpha}, 
					a_{1}\join(\beta\meet\comp\alpha)]\\
					&=[(a_{0}\meet\alpha)\join(a_{0}\meet\comp\beta), 
					(a_{1}\join\comp\alpha)\meet(a_{1}\join\beta)]\\
					a_{0}\meet\alpha&=a_{0}\meet(a_{0}\join\comp a_{1})=a_{0}\\
					a_{1}\join\comp\alpha&=a_{1}\join(\comp a_{0}\meet a_{1})=a_{1}\\
					\text{and so }
					[(a_{0}\meet\alpha)\join(a_{0}\meet\comp\beta), 
					(a_{1}\join\comp\alpha)\meet(a_{1}\join\beta)]&=
					[a_{0}\join(a_{0}\meet\comp\beta), 
					a_{1}\meet(a_{1}\join\beta)]\\
					&=a.
				\end{align*}
		
			\item  
			\begin{align*}
				(a\Rightarrow b)\Rightarrow b&=[b_{0}\meet\comp{\alpha}, 
					b_{1}\join\alpha]\Rightarrow b\\
					\ell(a\Rightarrow b)&=\beta\meet\comp\alpha.\\
					\text{Thus }
					[b_{0}\meet\comp{\alpha}, 
					b_{1}\join\alpha]\Rightarrow b&=
					[b_{0}\meet\comp{\beta\meet\comp\alpha}, 
					b_{1}\join(\beta\meet\comp\alpha)]\\
					b_{0}\meet\comp\beta&=b_{0}\meet(\comp b_{0}\meet b_{1})=0\\
					b_{1}\join\beta&=b_{1}\join b_{0}\join\comp b_{1}=1\\
					\text{ and so }
					[b_{0}\meet\comp{\beta\meet\comp\alpha}, 
					b_{1}\join(\beta\meet\comp\alpha)]&=
					[b_{0}\meet\alpha, b_{1}\join\comp\alpha]\\
					&=b*a.
			\end{align*}
		
			\item  
			\begin{align*}
				a\Rightarrow (b\Rightarrow c)&=
				a\Rightarrow[c_{0}\meet\beta, c_{1}\join\comp\beta]\\
				&=[(c_{0}\meet\beta)\meet\alpha, (c_{1}\join\comp\beta)\join\comp\alpha].
			\end{align*}
			Since this is symmetric in $\alpha$ and $\beta$ we have the result.
		\end{enumerate}
\end{proof}

We also note that $a\Rightarrow a =\Delta(a\join a, a)\rightarrow 
a=a\rightarrow a=\one$ and so the $1$ of $\mathcal L/\simeq $ is 
$\eqcl[\one]$. 

Lastly we have meets as expected. Firstly note that
$a\caret b\le a$ and is $\le \Delta(a\join b, b)\simeq b$. And as 
$a\le b$ implies $a\Rightarrow b=\Delta(a\join b, a)\rightarrow b = 
\Delta(b, a)\rightarrow b =\one$ (as $\Delta(b, a)\le b$) we see that 
$\eqcl[a]\le\eqcl[b]$.
Thus $\eqcl[a\caret b]$ is a lower bound to both $\eqcl[a]$ and $\eqcl[b]$.
Next we want to compute the meet. There are several steps:
\begin{lem}
	$\one\Rightarrow a=\one$.
\end{lem}
\begin{proof}
	$\ell(\one)=0$ and so $\one\Rightarrow a=[a_{0}\join 0, a_{1}\meet 1]=a$.
\end{proof}

\begin{lem}
	$(a\Rightarrow b)*(b\Rightarrow a)=\one$.
\end{lem}
\begin{proof}
	Let $\alpha=\ell(a)$ and $\beta=\ell(b)$.
	\begin{align*}
		a\Rightarrow b&=[b_{0}\meet\comp\alpha, b_{1}\join\alpha]\\
		b\Rightarrow a&=[a_{0}\meet\comp\beta, a_{1}\join\beta]\\
		\ell(b\Rightarrow a)&=(a_{0}\meet\comp\beta)\join(\comp 
		a_{1}\meet\comp\beta)=\alpha\meet\comp\beta.\\
		\intertext{Hence }
		(a\Rightarrow b)*(b\Rightarrow 
		a)&=[b_{0}\meet\comp\alpha\meet\alpha\meet\comp\beta, 
		b_{1}\join\alpha\join\comp\alpha\join\beta]\\
		&=[0, 1].
	\end{align*}
\end{proof}

\begin{thm}
	$\eqcl[a]\meet\eqcl[b]=\eqcl[a\caret b]$.
\end{thm}
\begin{proof}
	As in any implication algebra we compute
	$$
	\bigl((a\Rightarrow(a\caret b))*(b\Rightarrow(b\caret 
	a))\bigr)\Rightarrow(a\caret b).
	$$
	\begin{align*}
		a\Rightarrow(a\caret b)&=[a_{1}\meet(a_{0}\join\comp 
		b_{1})\meet\comp\alpha, a_{0}\join(a_{1}\meet\comp 
		b_{0})\join\alpha]\\
		&=[a_{1}\meet\comp a_{0}\meet\comp b_{1}, a_{0}\join\comp 
		a_{1}\join\comp b_{1}]\\
		&=[\comp\alpha\meet\comp b_{1}, \alpha\join\comp b_{0}]\\
		&=\Delta(\one, a\Rightarrow b).\\
		\intertext{Likewise we have }
		b\Rightarrow(b\caret a)&=\Delta(\one, b\Rightarrow a)\\
		\intertext{and hence }
		(a\Rightarrow(a\caret b))*(b\Rightarrow(b\caret a))&=
		\Delta(\one, (a\Rightarrow b)*(b\Rightarrow a))\\
		&=\Delta(\one, \one)=\one.\\
		\intertext{Therefore}
		\bigl((a\Rightarrow(a\caret b))*(b\Rightarrow(b\caret 
		a))\bigr)\Rightarrow(a\caret b)&=\one\Rightarrow(a\caret b)\\
		&=a\caret b.
	\end{align*}
\end{proof}

Thus we have established that $\mathcal L/\simeq$ is an implication 
lattice with the following operations:
\begin{align*}
    \one &= [\one]\\
    \eqcl[a]\meet\eqcl[b] & =\eqcl[a\caret b]  \\
    \eqcl[a]\join\eqcl[b] & =\eqcl[a * b]  \\
    \eqcl[a]\rightarrow\eqcl[b] & =\eqcl[a\Rightarrow b]. 
\end{align*}

This implication algebra is very closely tied to $\mathcal L$. 
Locally it is exactly $\mathcal L$ as the next theorem shows us.

\begin{thm}
    On each interval $[a, \one]$ in $\mathcal L$ the mapping 
$x\mapsto\eqcl[x]$ is an implication embedding.
\end{thm}

We prove this via another short series of lemmas.

\begin{lem}
	Let $\mathcal L$ be a cubic algebra and $a\in\mathcal L$. Then on $[a, 
	\one]$ we have
	\begin{align*}
		b*c&=b\join c\\
		b\dblCaret c&=b\meet c\\
		b\Rightarrow c&=b\rightarrow c.		
	\end{align*}
\end{lem}
\begin{proof}
	Without loss of generality we are in an interval algebra, and $a=[0, 0]$.
	\begin{align*}
		b*c&=[0, b]*[0, c]\\
		&=[0\meet\comp c, b\join c]\\
		&=[0, b\join c]\\
		&=b\join c.\\
		b\dblCaret c&=b\meet c&&\text{ as the meet exists.}\\
		b\Rightarrow c&=[0\meet b, c\join\comp b]\\
		&=[0, c\join\comp b]\\
		&=b\rightarrow c.
	\end{align*}
\end{proof}

\begin{lem}
	Let $\mathcal L$ be a cubic algebra and $a\in\mathcal L$. If $b, 
	c\geq a$ then
	$$
	b\simeq c\iff b=c.
	$$
\end{lem}
\begin{proof}
	If $b=\Delta(b\join c, c)$ then we have 
	$a\le c$ and $a\le b= \Delta(b\join c, c)$ and so 
	$b\join c = a\join\Delta(b\join c, a)\le c\join c=c$. Likewise 
	$b\join c\le b$ and so $b=c$.
\end{proof}

We note that a small variation of the proof shows that if $a\le b, c$ 
then $b\preceq c$ iff $b\le c$.

What does all this say about caret? We've seen that caret collapses 
to meet. And that locally an MR-algebra is like an implication 
algebra. In some sense we can then view the MR-algebra as a family of 
connected
implication algebras where $\Delta$ describes the interaction between 
the pieces. The operation $*$ describes the way the pieces link 
together (see \cite{BO:eq} theorem 4.6) and caret is then describing 
local
self-similarity in the following sense -- the interval
$]\leftarrow,  a\join b]$ has intervals $]\leftarrow,  a\caret b]$
which is similar via $\Delta$ to $]\leftarrow,  b\caret a]$ and these 
similar intervals occur densely.


\end{document}